\documentclass[12pt]{article}
\usepackage{geometry}
\usepackage{amsmath,amsthm,amscd,amssymb,latexsym, enumerate,bm,lscape,epsfig,nicefrac,hyperref,amsfonts,cite,mathtools,mathrsfs}

\newcommand{\N}{\mathbb{N}}

\newcommand{\R}{\mathbb{R}}
\newcommand{\C}{\mathbb{C}}
\newcommand{\E}{\mathbb{E}}

\newcommand{\Tr}{\text{\rm Tr}}

\newtheorem{info}{}
\newtheorem{theorem}[info]{Theorem}
\newtheorem{corr}[info]{Corollary}

\newtheorem{lem}[info]{Lemma}
\newtheorem{prop}[info]{Proposition}

\newtheorem{conjecture}[info]{Conjecture}
\numberwithin{info}{section}

\numberwithin{equation}{section}
\renewcommand{\[}{\begin{equation}}
	\renewcommand{\]}{\end{equation}}
\makeatletter

\makeatletter
\g@addto@macro\normalsize{%
	\setlength\abovedisplayskip{5pt}
	\setlength\belowdisplayskip{5pt}
	\setlength\abovedisplayshortskip{4pt}
	\setlength\belowdisplayshortskip{4pt}}
\makeatother

\newcommand{\lam}{\lambda}

\renewcommand{\P}{\mathbb{P}}
\renewcommand{\cal}{\mathcal}

\newcommand{\sq}{\sqrt}
\newcommand{\sign}{\mathrm{Sign}}

\renewcommand{\sq}{\sqrt}

\newcommand{\Sum}{\mathrm{\Sum}}

\renewcommand{\Im}{\mathrm{Im}}
\renewcommand{\Re}{\mathrm{Re}}
\renewcommand{\Im}{\mathrm{Im}}
\renewcommand{\Re}{\mathrm{Re}}

\newcommand{\erfc}{\mathrm{erfc}}

\newcommand{\pf}{\mathrm{Pf}}

\renewcommand{\l}{\left}
\renewcommand{\r}{\right}

\newcommand{\Sb}{\bar{S}}

\begin{document}
	\title{Moments of Characteristic Polynomials of Non-Symmetric Random Matrices}
	
	\author{
		Pax Kivimae \thanks{Courant Institute of Mathematical Sciences, New York University. Email: pax.kivimae@cims.nyu.edu. This research was partially supported by NSF grant DMS-2202720.}}
	
	\maketitle

	\begin{abstract}
		We study the moments of the absolute characteristic polynomial of the real elliptic ensemble, including the case of the real Ginibre ensemble. We obtain asymptotics for all integral moments inside the real bulk to order $1+o(1)$. In particular, for the real Ginibre ensemble, this extends known computations for even moments, and confirms a recent conjecture of Serebryakov and Simm \cite{nickrecent} in the integral case. For the elliptic case, this generalizes computations of first two moments by Fyodorov \cite{complexity-fyodorov-non-gradient} and Fyodorov and Tarnowski \cite{fyo-condition-elliptic}. We additionally find uniform asymptotics for the multi-point correlations of the absolute characteristic polynomial. Our proof relies on a relation between expectations for the absolute characteristic polynomial and the real correlation functions, as well as an algebraic method of obtaining asymptotics for the behavior of these correlation functions near the diagonal.
	\end{abstract}
	
	\maketitle
	
	\section{Introduction}
	
	The characteristic polynomials of random matrices have long played the role of both an important and difficult to understand object. Their properties have seen intense study by researchers in a number of different areas in both mathematics and physics due to the sheer number of applications and connections that have been found. Among these are those in number theory \cite{numbertheory-review}, quantum chaos \cite{qc1}, statistical mechanics and disordered systems \cite{tribe1,mp2,yan1}, and statistics and machine learning \cite{ml-complexity-1,replica-tensor-PCA,fyo-stat-2}, not to mention the applications to study the random matrices themselves \cite{fyo-condition-real,fyo-condition-elliptic, BrezinBegin,BrezinReal}.
	
	Of particular complexity are the higher moments of the absolute (or non-absolute) characteristic polynomial. However, despite extensive study, many results are only known for ensembles with eigenvalues on the real line or unit circle. Moreover, most of the more general results have primarily been restricted to the case of complex random matrices. 
	
	The purpose of this paper is to evaluate the higher integral moments of the absolute characteristic polynomial for a family of non-symmetric real matrix ensembles on the real bulk. Specifically, we focus on the family of real elliptic ensembles \cite{sommers-elliptic}. This family includes the real Ginibre ensemble, where our results extend the recent computation of Serebryakov and Simm \cite{nickrecent} for even moments, and confirms their conjecture in the case of odd moments. In this general elliptic case, our computations are the first above the first two moments (as detailed below), and reveal a simple connection between the asymptotics in the general elliptic case and the special case of the real Ginibre ensemble, suggesting some degree of universality.
	
	We now mention two motivations for studying the elliptic ensemble, outside of the fact that it contains the fundamental case of the real Ginibre ensemble. The first is that outside of the special case of the real Ginibre ensemble, the real elliptic ensemble lacks bi-orthogonal symmetry, which removes most of the number of tools available for their study. Despite this, they are expected to retain most properties of the real Ginibre ensemble, making them a useful test case when approaching universality for questions involving non-symmetric real random matrices. The second is that the ensemble has appeared quite frequently in applications. In particular, in situations where non-gradient models or partially symmetric interactions occur, elliptic random matrices have often played a role analogous to that of the Gaussian orthogonal ensemble in the gradient case, with the asymmetry parameter tuned to the model in question. Classical examples appear in the study of neural networks \cite{original-glassy,original-1,ml-complexity-1} and ecology \cite{may,fyo-maywigner-1,fyo-maywigner-2}. 
	
	\subsection{Result}
	
	To state our result, we recall the definition of the real elliptic ensemble, which depends on a symmetry parameter $\tau\in (-1,1)$. This is given as distribution over $N$-by-$N$ real matrices with (Lebesgue) density given by
	\[\P_N(A)=\frac{1}{Z_{N,\tau}}\exp\left(\frac{-N}{2(1-\tau^2)}\left(\Tr(AA^t)-\tau \Tr(A^2)\right)\right),\]
	where $Z_{N,\tau}$ is a normalization constant. We will denote this ensemble as $GEE(\tau,N)$, and denote a matrix sampled from it as $A_{N,\tau}$.
	
	The limiting law for the empirical measure for the elliptic ensemble was identified by Girko \cite{girko1} (see as well \cite{naumov}) who found it converged to the uniform measure on the ellipse $\cal{E}_\tau\subseteq \C$ given by
	\[\cal{E}_\tau=\left\{x+iy\in \C:\frac{x^2}{(1+\tau)^2}+\frac{y^2}{(1-\tau)^2}\le 1\right\}.\label{eqn:ellipse}\]
	In the case of $\tau=0$, this recovers the circular law for the Ginibre ensemble. If one sends $\tau \to 1$, this converges to semicircle law on the real line, formally recovering the result for the Gaussian orthogonal ensemble. The intersection of the bulk of the spectrum with the real line is thus $(-(1+\tau),1+\tau)$. Our first main result is the following asymptotics.
	
	\begin{theorem}
		\label{theorem: moments of real bulk}
		Let us fix $\tau\in [0,1)$, $\mu\in (-(1+\tau),1+\tau)$, and $\ell \in \N$. Then we have that
		\[
		\E|\det(A_{N,\tau}-\mu I)|^{\ell}=e^{\frac{N \ell}{2} \left(\frac{\mu^2}{1+\tau}-1\right)}N^{\ell(\ell-1)/4}C_{\tau}(\ell)\l(1+O\l(N^{-1}\r)\r).
		\]
		where here $C_{\tau}(\ell)$ is the constant
		\[C_{\tau}(\ell)=\left(\frac{(1+\tau)^{\ell/2}}{(1-\tau^2)^{\ell(\ell+1)/4}}\right)\left(\frac{(2\pi)^{\ell/2}}{2^{\ell(\ell-1)/4}\prod_{j=1}^{\ell}\Gamma(j/2)}\right).\]
		\noindent
		Moreover, this estimate is uniform in $\mu \in (-(1+\tau)+\epsilon,1+\tau-\epsilon)$ for any fixed $\epsilon>0$.
	\end{theorem}
	
	Based on this result, we conjecture some asymptotic formulas for the fractional moments. This extends the recent conjecture of Serebryakov and Simm \cite{nickrecent} on the moments of the real Ginibre ensemble to the general elliptic case.
	
	\begin{conjecture}
		\label{conjecture: moment conjecture}
		Fix $\tau\in (-1,1)$. For any $\ell\in (-1,\infty)$ and $\mu \in (-(1+\tau),1+\tau)$, we have that
		\[
		\E|\det(A_{N,\tau}-\mu I)|^{\ell}=e^{\frac{N \ell}{2} \left(\frac{\mu^2}{1+\tau}-1\right)}N^{\ell(\ell-1)/4}C_{\tau}(\ell)(1+o(1)),
		\]
		where here
		\[C_{\tau}(\ell)=\left(\frac{(1+\tau)^{\ell/2}}{(1-\tau^2)^{\ell(\ell+1)/4}}\right)\left(\frac{(2\pi)^{\ell/2}G\l(\frac{1}{2}\r)}{2^{\ell(\ell-1)/4}G\l(\frac{\ell}{2}+1\r)G\l(\frac{\ell}{2}+\frac{1}{2}\r)}\right),\]
		where here $G$ denotes the Barnes $G$-function.
	\end{conjecture}
	
	We sketch in Appendix \ref{section:appendix:conjecture} how this is obtained from Theorem \ref{theorem: moments of real bulk} above, and in particular show why $C_{\tau}(\ell)$ coincides with the above definition given in Theorem \ref{theorem: moments of real bulk} for natural numbers.
	
	We note as well that our result yields a simple relationship between the moments of the general elliptic ensemble and the special case of the Ginibre ensemble. In particular we have the asymptotic
	\[\frac{\E|\det(A_{N,\tau}-\mu I)|^{\ell}}{\E|\det(A_{N,0}-(1+\tau)^{1/2}\mu I)|^{\ell}}=\left(\frac{(1+\tau)^{\ell/2}}{(1-\tau^2)^{\ell(\ell+1)/4}}\right)(1+o(1)),\]
	which is the same scaling on $\mu$ which relates the unit disk $\cal{E}_0$ to the ellipse $\cal{E}_\tau$. This is similar (and as we see below, related) to relations found in the asymptotics of the real correlation functions after re-scaling.
	
	We also compare our results to that of \cite{afanasievinterpolate,afanasievComplex}, which imply that for $\ell$ even and $z\in \C\setminus\R$ with $|z|<1$ we have that
	\[\E |\det(A_{N}-\mu I)|^{\ell}=N^{\ell^2/8}e^{N\frac{\ell}{2}\left(|z|^2-1\right)}\left(\frac{(2\pi)^{\gamma/4}}{G\left(1+\frac{\gamma}{2}\right)}\right),\]
	when $A_N$ is sampled from either the real or complex Ginibre ensemble. In particular, the exponent in $N$ changes from $\ell^2/8$ to $\ell(\ell-1)/4$ as you cross from $\C\setminus \R$ to $\R$ inside of the bulk, which is interesting, given the close relation of this factor to finer descriptions of the characteristic polynomial using log-correlated fields (see \cite{fyo-keating}).
	
	Next, we give a uniform result for the multi-point correlation functions for the absolute characteristic polynomial.
	
	\begin{theorem}
		\label{theorem: real bulk multi point}
		Let us fix $\tau\in [0,1)$, $\ell\in \N$, and $\epsilon>0$. Then there is small $c>0$ for such that for distinct $\mu_1,\cdots,\mu_{\ell}\in (-(1+\tau)+\epsilon,1+\tau-\epsilon)$ we the uniform bound
		\[
		\E\left[\prod_{i=1}^{\ell}\left|\det(A_{N,\tau}-\mu_i I)\right|\right]=e^{\frac{N }{2} \sum_{i=1}^{\ell}\left(\frac{\mu_i^2}{1+\tau}-1\right)}N^{\ell(\ell-1)/4}\left(\frac{(4\pi)^{\ell/2}(1+\tau)^{\ell/2}}{(1-\tau^2)^{\ell(\ell+1)/4}}\right)\times\]
		\[\left(\frac{\pf\l(\l[\cal{H}\l(\frac{\sqrt{N}\left(\mu_i-\mu_j\right)}{\sqrt{1-\tau^2}}\r)\r]_{1\le i,j\le \ell}\r)}{\Delta\left(\frac{\sqrt{N}\mu_1}{\sqrt{1-\tau^2}},\cdots,\frac{\sqrt{N}\mu_\ell}{\sqrt{1-\tau^2}}\right)}\l(1+O(N^{-1})\r)+O(e^{-Nc})\right),
		\]
		where here $\Delta$ denotes the Vandermonde determinant and $\cal{H}:\R \to M_2(\R)$ is given by  
		\[\cal{H}(x):=\frac{1}{\sqrt{2\pi}}\begin{bmatrix}
			\sqrt{\left(\frac{\pi}{2}\right)}\sign(x)-\int_0^x e^{-\frac{t^2}{2}}dt&  e^{-\frac{x^2}{2}}\\
			- e^{-\frac{x^2}{2}}& -x e^{-\frac{x^2}{2}}
		\end{bmatrix}.\label{eqn:def:H-kernel}\]
	\end{theorem}
	
	It is interesting to compare this to the local multi-point correlation functions in the case of $\tau=0$ for even $\ell$ given in \cite{nickrecent}. Here is is shown that if $\ell$ is even, then for $\mu\in (-(1+\tau),1+\tau)$, and distinct $\zeta_1,\cdots,\zeta_\ell\in \C$, one has that
	\[
	\E\left[\prod_{i=1}^{\ell}\l|\det(A_{N,0}-\left(\mu+\frac{\zeta_i}{\sqrt{N}}\right)I)\r|\right]=
	\]
	\[
	e^{\frac{N \ell}{2} \left(\mu^2-1\right)+\sqrt{N}\mu \sum_{i=1}^{\ell} \zeta_i}(2\pi)^{\ell/4}\frac{\pf([(\zeta_j-\zeta_i)e^{\zeta_i\zeta_j}]_{1\le i,j\le \ell})}{\Delta(\zeta_1,\cdots,\zeta_\ell)}\l(1+o(1)\r).\label{eqn:tribe:kernel:complex}
	\]
	It would be interesting to know how these expressions relate, or even how one could obtain (\ref{eqn:tribe:kernel:complex}) directly from (\ref{eqn:def:H-kernel}).
	
	\subsection{Related Results}
	
	We now comment on a number of related results. For the real elliptic ensemble, asymptotics for the first moment of the absolute characteristic polynomial on the real bulk were computed by Fyodorov \cite{complexity-fyodorov-non-gradient}, and follow from its relation (see \cite{fyo-maywigner-1}) to the mean-density function, which was first computed by Forrester and Nagao \cite{forrester-elliptic}\footnote{Strictly speaking, they only establish the case of even $N$, with the extension to $N$ odd case following from methods in \cite{forrester-odd,sinclair}.}. Asymptotics for the second moment were derived by Fyodorov and Tarnowski \cite{fyo-condition-elliptic} while studying the condition numbers of such matrices. We mention as well the earlier work of Akemann, Phillips, and Sommers \cite{akemann1} who found an exact representation for the expectation of the product of the (non-absolute) characteristic polynomial at two distinct points, though they did not consider its asymptotics. 
	
	In the case of the real Ginibre ensemble, the asymptotics for the first moment in the bulk were obtained by Edelman, Kostlan and Shub \cite{edelman1}, who used it to compute the mean-density function, with similar results for the second moment given in the follow up work of Edelman \cite{edelman2}. 
	
	As mentioned above, results for all even moments in the real Ginibre case are also known as well. In particular, Afanasiev computed the asymptotics the even moments of a matrix of real i.i.d entries (a generalization of the real Ginibre ensemble) in the real bulk in \cite{afanasievreal} up to an unknown constant term, building upon a similar result of theirs in the complex case \cite{afanasievComplex}. Building upon the work of Tribe and Zaboronski \cite{tribe1,tribe3}, this constant was then computed by Serebryakov and Simm \cite{nickrecent}.
	
	To the knowledge of the author, our results are the first asymptotics for the odd moments of the absolute characteristic polynomial for any real non-symmetric matrix above the first, with one notable exception: The absolute determinant of the real Ginibre ensemble itself, for which the asymptotics of all fractional moments may be computed by using a classical relation between it and a ratio of Selberg integrals (see \cite{nickrecent} or (\ref{eqn:determinant of ginibre ratio}) below).
	
	We now comment on some other results with relation to this case. Sommers and Khoruzhenko \cite{khoruzhenko-matrix-averages} were able to obtain explicit formulas for the expectation of Schur functions of the eigenvalues of a real Ginibre matrix at fixed $N$. Forrester and Rains \cite{forrester-matrix-averages} expanded upon this work, and used it to give an explicit duality formula between integer moments of the characteristic polynomials and a finite-dimensional integral, to which one may easily apply Laplace's method. However, due to the nature of the method it cannot be generalized to absolute moments, yielding results only for the even moments. 
	
	Such duality formulas had previously been found in the case of unitary-invariant matrix ensembles by Fyodorov and Khoruzhenko \cite{uni-hermite-fyodorov}, who studied even powers of the absolute characteristic polynomial, again using Schur functions. The generality of this result is quite beautiful, though at the same time it helps to explain the difficulty that the lack of bi-orthogonal symmetry poses for analysis of the elliptic ensemble outside of the special case of the real Ginibre ensemble.
	
	Finally, for the complex Ginibre ensemble, significantly more is known. In particular, Webb and Wong \cite{webbwong} have computed asymptotics for all fractional moments inside of the bulk. This result however, relies on Riemann-Hilbert methods which are exclusive to the complex case. Moreover, a number of duality formulas in the case for integral moments of the non-absolute characteristic polynomials in elliptic ensembles are given by Akemann and Vernizzi in \cite{moments-and-duality-ginibre-elliptic}.
	
	Lastly, we highlight an application of this work as presented in our companion paper \cite{kivimae}. In this work, we study the number of zeros for a certain random vector field on the sphere, known as the asymmetric spherical $p$-spin glass model \cite{original-glassy,complexity-fyodorov-non-gradient}. Such counts are the focus of the theory of landscape complexity, and their relation to absolute determinants of random matrices comes from the Kac-Rice formula (see the work of Fyodorov \cite{fyo1} for an introduction to this area). In particular, our main result, Theorem \ref{theorem: moments of real bulk} is applied in \cite{kivimae} to show the total number of zeros concentrates on its average. This result proves to be a crucial step to provide the first rigorous example of a system with a transition from relative to absolute instability, as introduced by Ben Arous, Fyodorov and Khoruzhenko \cite{fyo-maywigner-1,fyo-maywigner-2}. 
	
	\subsection{Method}
	
	Our approach will rely on relations between the absolute characteristic polynomial and the (eigenvalue) correlation functions for the ensemble. Such relations have proven to be a crucial tool in the study of the Hermitian and symmetric ensembles, leading back to the fundamental works of Br\'{e}zin, Fyodorov, Hikami, Mehta, Normand, and Strahov \cite{BrezinBegin,KeyRelationPhys,FyoKeyRelation} (see as well \cite{holgersym,holgercom}). These relations rely on the invariant nature of the ensemble, and in-particular, on the occurrence of the Vandermonde determinant in the eigenvalue probability distribution function. The explicit result is given in Lemma \ref{lem:tuca-relation} below. 
	
	Asymptotics for the real correlation function are known \cite{forrester-elliptic,forrester-odd,sinclair}, and may be obtained through employing a Pfaffian representation of the correlation functions for finite $N$. In particular, after employing these asymptotics one could obtain a local version of Theorem \ref{theorem: real bulk multi point} pointwise.
	
	The real work is to obtain Theorem \ref{theorem: moments of real bulk}. One essentially wants to take all the points in Theorem \ref{theorem: real bulk multi point} to coincide. As is clear the division by the Vandermonde determinant on the right-hand side of Theorem \ref{theorem: real bulk multi point}, computing this limit is not immediately clear.
	
	In the case of even moments of the Ginibre ensemble (see \cite{nickrecent,tribe1,tribe3,afanasievreal}), the key is that one may write the kernel on the right-hand side of (\ref{eqn:tribe:kernel:complex}) as a certain expectation with respect to the $CSE(\ell)$, which is uniformly defined over all $(\zeta_1,\cdots,\zeta_\ell)$, and so has a clear limit if $\zeta_1=\cdots =\zeta_\ell=0$, which may then be computed using classical integral formulas.
	
	In our case however, we are not aware of such an integral representation, and so instead employ a more algebraic approach. In particular, we show that the limit as $\zeta_1,\cdots ,\zeta_\ell\to 0$ coincides with the same limit taken on a certain sum of partial derivatives of the correlation function. This follows from a general result on anti-symmetric holomorphic functions, which we modify to employ our  correlation functions. In particular, this not only gives a uniform expression for our limiting expression, but also reduces us to establishing asymptotics for the correlation functions in a suitable Sobolev norm, which can be reduced to obtaining similar asymptotics for the Pfaffian kernel.
	
	We note that we also believe our approach could immediately be generalized to compute the even moments of the absolute characteristic polynomial in the complex bulk, though as more approaches for the even case are known, we have not pursued this.
	
	Finally, we outline briefly the structure of the article. In Section \ref{section:correlation functions and characteristic polynomials}, we recall some basic facts about the ensemble, and provide our relation between the absolute characteristic polynomial and the real correlation function. In Section \ref{section:proof of main theorem}, we recall the Pfaffian kernel representation of the real correlation functions, which we use to prove our main results, up to a result on the asymptotics for the Pfaffian kernel. This result is shown in Section \ref{section:pfaffian kernel}, where we also give the exact formula for the Pfaffian kernel. Finally, in Appendix \ref{section:appendix:conjecture}, we provide the aforementioned statements on Conjecture \ref{conjecture: moment conjecture}.

	\section{Characteristic Polynomials and Correlations \label{section:correlation functions and characteristic polynomials}}
	
	To begin, we first recall the structure of the eigenvalue density for the real elliptic ensemble. Its complexity is somewhat surprising given the naturality of the model, and doubly so since the expressions for the density in the complex and quaternionic counterparts for the Ginibre ensemble (i.e. $\tau=0$) were already found in the original paper of Ginibre \cite{ginibre-12}. Indeed, it was only over twenty years later that an expression was derived by Lehmann and Sommers \cite{sommers-elliptic}, who also computed  the density of the elliptic ensemble as well.
	
	This difficulty may be partially explained by the fact that this ensemble lacks a continuous eigenvalue density function on $\C$. Indeed, for each $K\in \N$ such that $N-K$ is non-negative and even, the subset of $N$-by-$N$ matrices such that $A_{N}$ has precisely $K$ real eigenvalues has non-zero Lebesgue measure. We note the evenness of $N-K$ is clear, as eigenvalues of $A_N$ are defined as the solutions to the equation $\det(A_N-\lam)=0$, which has only real coefficients, so that the non-real eigenvalues must also come in conjugate pairs. In particular, when $A_N$ is sampled from $GEE(\tau,N)$, this $K$ is a random variable. However, when one conditions on the value of $K$, the eigenvalue density has an elegant expression.
	
	To give this, let us denote the event that $A_N$ has exactly $k$ real eigenvalues as $\cal{E}_{N,k}$. Conditional on $\cal{E}_{N,k}$, let us denote the $k$ real eigenvalues of $A_N$, as $\lam_1,\dots \lam_k$. Defining $m=(N-k)/2$, we denote the $m$-eigenvalues which live in the upper half-plane as $z_1,\dots z_{m}\in \mathbb{H}$, so that in total the eigenvalues of $A_N$ are given by $\lam_1,\dots,\lam_k,z_1,\dots z_m, \bar{z}_1,\dots \bar{z}_m$. 
	
	Then the (symmetrized) conditional density of $(\lam_1,\dots,\lam_k,z_1,\dots z_m)$ on $\R^k\times \mathbb{H}^{m}$ is given by (see (2.7) of \cite{forrester-elliptic})
	\footnote{Note in the notation of \cite{forrester-elliptic} we have $b=N(1+\tau)^{-1}$.}
	\[P_{N,k,m}(\{\lam_i\}_{i=1}^k,\{z_i\}_{i=1}^{m}):=\frac{1}{K_N(\tau)}|\Delta(\{\lam_i\}_{i=1}^k\cup\{z_i\}_{i=1}^m\cup \{\bar{z}_i\}_{i=1}^m)|\times\]\[\frac{1}{k!m!}\prod_{i=1}^ke^{-\frac{N}{2(1+\tau)}\lam_i^2}\prod_{i=1}^{m}2e^{-\frac{N}{(1+\tau)}\Re(z^2)}\erfc \left(\Im(z_i)\sq{\frac{2N}{1-\tau^2}} \right),\]
	and $K_N(\tau)$ is the normalization constant
	\[K_N(\tau):=N^{-N(N+1)/4}(1+\tau)^{N/2}2^{N(N+1)/4}\prod_{i=1}^{N}\Gamma(i/2).\]
	From this we define the real $\ell$-point correlation function $\rho^{\ell}_N:\R^\ell\to \R$ by
	\[\rho^{\ell}_N(w_1,\dots, w_\ell):=\frac{1}{K_N(\tau)}\sum_{L+2M=N,\; L\ge \ell,M\ge 0}\frac{L!}{(L-\ell)!}\times \]\[\int_{\R^{L-\ell}\times \mathbb{H}^{M}}P_{N,L,M}(\{w_i\}_{i=1}^{\ell}\cup\{\lam_i\}_{i=1}^{L-\ell}, \{z_i\}_{i=1}^{M})\prod_{i=1}^{L-\ell}d\lam_i\prod_{i=1}^{M}d^2z_i.\label{eqn:correlation density expression}\]
	We note that while the variable $M=(N-L)/2$ here is redundant, its inclusion makes the expressions below significantly more manageable. We also note for the clarity of the reader, this coincides with the alternative standard definition of a correlation function (see (5.1) of \cite{borodin-elliptic} or (4.2) of \cite{forrester-elliptic}) by Corollary 2 of \cite{borodin-elliptic}.
	
	We now demonstrate a relationship between the multi-point correlations of the absolute characteristic polynomial and the values of the real correlation function for a larger matrix.
	
	\begin{lem}
		\label{lem:tuca-relation}
		For any $\ell\ge 1$, $N>\ell$, and distinct $\mu_1,\cdots,\mu_{\ell}\in \R$, we have that
		\[\E\left[\prod_{i=1}^{\ell}|\det(A_N-\mu_iI)|\right]=D_{N,\ell}(\tau)\frac{e^{\frac{N}{2(1+\tau)}\sum_{i=1}^{\ell}\mu_i^2}}{|\Delta(\mu_1,\dots \mu_\ell)|}\rho_{N+\ell}^{\ell}(\tilde{\mu}_1,\dots,\tilde{\mu}_\ell),\label{eqn:lem:tuca-relation}\]
		where $\tilde{\mu}_i:=\mu_i\sq{N/(N+\ell)}$ and 
		\[D_{N,\ell}(\tau):=\frac{K_{N+\ell}(\tau)}{K_N(\tau)}\left(\frac{N+\ell}{N}\right)^{N/2+(N+\ell)(N+\ell-1)/4}.\]
	\end{lem}
	\begin{proof}
		We note that for any $u_1,\dots u_i,v_1,\dots v_j\in \C$ we have that
		\[|\Delta(u_1,\dots, u_i,v_1,\dots, v_j)|=|\Delta(u_1,\dots, u_i)\Delta(v_1,\dots, v_j)|\prod_{k=1}^i\prod_{l=1}^{j}\l|u_k-v_l\r|.\]
		Applying this relation to (\ref{eqn:correlation density expression}) we obtain that
		\[\E\left[\prod_{i=1}^{\ell}|\det(A_N-\mu_i I)|\right]e^{-\frac{N}{2(1+\tau)}\sum_{i=1}^{\ell}\mu_i^2}|\Delta(\mu_1,\dots \mu_\ell)|=\]
		\[\sum_{L,M\ge 0, L+2M=N}\frac{1}{L!M!}\frac{1}{K_N(\tau)}\int_{\R^L\times \mathbb{H}^M}|\Delta(\{\mu_i\}_{i=1}^{\ell}\cup\{\lam_i\}_{i=1}^{L}\cup\{z_i\}_{i=1}^{M}\cup\{\bar{z}_i\}_{i=1}^{M})|\times\]
		\[\prod_{i=1}^{\ell}e^{-\frac{N}{2(1+\tau)}\mu_i^2}\prod_{i=1}^{L}e^{-\frac{N}{2(1+\tau)}\sum_{i=1}^{L}\lam_i^2}\prod_{i=1}^{M}2e^{-\frac{N}{(1+\tau)}\Re(z^2)}\erfc \left(\Im(z_i)\sq{\frac{2N}{1-\tau^2}} \right)\prod_{i=1}^{L}d\lam_i\prod_{i=1}^{M}d^2 z_i.\label{eqn:ignore-655}\]
		Rescaling all integration variables by a factor of $\sq{N/(N+\ell)}$, we see that the right hand side of (\ref{eqn:ignore-655}) may be rewritten as
		\[D_{N,\ell}(\tau)\sum_{L,M\ge 0, L+2M=N}\frac{(L+\ell)!}{L!}\times\]
		\[\int_{\R^L\times \mathbb{H}^M}P_{N+\ell,L+\ell,M}(\{\tilde{\mu}_i\}_{i=1}^{\ell}\cup\{\lam_i\}_{i=1}^{L}\cup\{z_i\}_{i=1}^{M}\cup\{\bar{z}_i\}_{i=1}^{M})\prod_{i=1}^{L}d\lam_i\prod_{i=1}^{M}d^2 z_i.\]
		Re-indexing, we see this coincides with $D_{N,\ell}(\tau)\rho^{\ell}_{N+\ell}(\tilde{\mu}_1,\dots \tilde{\mu}_k)$, which completes the proof.
	\end{proof}
	
	Finally, we will complete this with the asymptotics of $D_{N,\ell}(\tau)$ for large $N$.
	
	\begin{lem}
		\label{lem:limit for D}
		For any $\ell\in \N$ we have that
		\[D_{N,\ell}(\tau)=\left(1+O(N^{-1})\right)N^{-\frac{\ell}{2}}e^{-N\ell/2}\left(4\pi (1+\tau)\right)^{\ell/2}.\]
	\end{lem}
	\begin{proof}
		To begin, we note that
		\[\frac{K_{N+\ell}(\tau)}{K_{N}(\tau)}=\frac{(N+\ell)^{-(N+\ell)(N+\ell+1)/4}2^{(N+\ell)(N+\ell+1)/4}}{N^{-N(N+1)/4}2^{N(N+1)/4}}(1+\tau)^{\ell/2}\prod_{i=N+1}^{N+\ell}\Gamma(i/2)=\]
		\[\left(\frac{N+\ell}{N}\right)^{-(N+\ell)(N+\ell+1)/4}N^{-N\ell/2-(\ell+1)\ell/4}2^{N\ell/2+(\ell+1)\ell/4}(1+\tau)^{\ell/2}\prod_{i=N+1}^{N+\ell}\Gamma(i/2).\]
		First off, we combine the factors in $D_{N,\ell}(\tau)$ involving $(\frac{N+\ell}{N})$. This gives \[\left(\frac{N+\ell}{N}\right)^{-(N+\ell)(N+\ell+1)/4+N/2}\left(\frac{N+\ell}{N}\right)^{(N+\ell)(N+\ell-1)/4}=\left(\frac{N+\ell}{N}\right)^{-\ell/2}=1+O(N^{-1}),\]
		so that
		\[D_{N,\ell}(\tau)=\left(1+O\left(N^{-1}\right)\right)N^{-N\ell/2-(\ell+1)\ell/4}2^{N\ell/2+(\ell+1)\ell/4}(1+\tau)^{\ell/2}\prod_{i=N+1}^{N+\ell}\Gamma(i/2).\]
		Now, for fixed $j\in \N$, we have by Stirling's approximation that
		\[\Gamma(N/2+j/2)=(1+O(N^{-1}))\left(\frac{4\pi}{N+j}\right)^{1/2}\left(\frac{N+j}{2e}\right)^{(N+j)/2}=\]
		\[\left(1+O(N^{-1})\right)(2\pi)^{1/2}2^{-(j-1)/2}\left(2e\right)^{-N/2}N^{(N+j-1)/2},\]
		where we have used in the final step that
		\[\left(\frac{N+j}{N}\right)^{(N+j)/2}=e^{\frac{j}{2}}(1+O(N^{-1})).\]
		From this we see that
		\[\prod_{i=N+1}^{N+\ell}\Gamma(i/2)=\left(1+O(N^{-1})\right)(2\pi)^{\ell/2}2^{-(\ell-1)\ell/4}(2e)^{-N\ell/2}N^{N\ell/2+\ell(\ell-1)/4}.\]
		In particular, we may simplify
		\[D_{N,\ell}(\tau)=\left(1+O(N^{-1})\right)N^{-\ell(\ell+1)/4+\ell(\ell-1)/4}2^{\ell(\ell+1)/4-\ell(\ell-1)/4}(1+\tau)^{\ell/2}(2\pi)^{\ell/2}e^{-N\ell/2}.\]
		Canceling gives the desired result.
	\end{proof}
	
	\section{Proof of Theorems \ref{theorem: moments of real bulk} and \ref{theorem: real bulk multi point} \label{section:proof of main theorem}}
	
	In this section we will prove our main results. We will begin by recalling the Pfaffian representation of the real correlation function and its asymptotics. From these we will immediately be able to prove Theorem \ref{theorem: real bulk multi point}. Then we proceed to develop the tools to also obtain Theorem \ref{theorem: moments of real bulk} from these asymptotics.
	
	To begin, we recall \cite{forrester-elliptic} that the real correlation function has a Pfaffian representation
	\[\rho^{\ell}_N(x_1,\dots,x_\ell)=N^{\ell/2}\pf([K_N(x_i,x_j)]_{1\le i,j\le \ell}),\label{eqn:even Pfaffian formula}\]
	where $K_N(x,y):\R^2\to M_2(\R)$ is a certain matrix kernel satisfying $K_N(x,y)=-K_N(y,x)^T$. Moreover, there is some $S_N:\R\times \R\to \R$ such that the kernel $K_N$ is given by
	\[K_N(x,y):=\begin{bmatrix}
		-I_N(x,y)& S_N(x,y)\\
		-S_N(y,x)& D_N(x,y)
	\end{bmatrix},\label{eqn:Pfaffian K-def}\]
	where
	\[D_N(x,y):=\partial_x S_N(x,y),\;\; I(x,y):=\frac{1}{2}\sign(y-x)-\int_x^y S_N(x,z)dz.\label{eqn:Pfaffian I and D def}\]
	The precise definition of $S_N$ is then given in Lemma \ref{lem:exact form of pfaffian kernel} below, which also proves these facts. However, to derive our main results, we only need two results. The first is simply that $S(x,y)$ extends to an entire function on $\C^2$, which is clear by inspection of Lemma \ref{lem:exact form of pfaffian kernel}. The section is an asymptotic formula for $K_N$ as $N\to \infty$.
	
	The first shows that the function $S_N(x,y)$ converges in the bulk to $\cal{S}_{\tau}(\sqrt{N}(x-y))$ where
	\[\cal{S}_{\tau}(x):=\frac{1}{\sq{2\pi(1-\tau^2)}}\exp \left(\frac{-x^2}{2(1-\tau^2)}\right).\]
	The asymptotic form of $K_N(x,y)$ is then given by $\cal{K}_{\tau}(\sqrt{N}(x-y))$, where $\cal{K}_{\tau}$ has essentially the same structure as $K_N$. In particular we have that
	\[\cal{K}_{\tau}(x):=\begin{bmatrix}
		-\cal{I}_{\tau}(x)& \cal{S}_{\tau}(x)\\
		-\cal{S}_{\tau}(x)& \cal{D}_{\tau}(x)
	\end{bmatrix},\]
	where
	\[\cal{D}_{\tau}(x):=\partial_x \cal{S}_{\tau}(x)=\frac{-x}{\sq{2\pi(1-\tau^2)^3}}\exp \left(\frac{-x^2}{2(1-\tau^2)}\right),\label{eqn:Pfaffian pure D}\]
	\[\cal{I}_{\tau}(x):=\frac{1}{2}\sign(-x)-\int_x^0 \cal{S}_{\tau}(x,z)dz=-\frac{1}{2}\sign(x)+\frac{1}{2}\int_0^{\frac{x}{\sqrt{(1- \tau
				^2)}}}e^{-\frac{t^2}{2}}dt,\label{eqn:Pfaffian pure I}\]
	Note that we have omitted the fixed $\tau$ from the notation for $K_N$, but included it in $\cal{K}_{\tau}$, the reason for which will become clear momentarily.
	
	The next result establishes that for $x,y\in (-1-\tau,1+\tau)$, the kernel $K_N(x,y)$ is exponentially close to $\cal{K}_{\tau}(\sq{N}(x-y))$, and is proven in Section \ref{section:pfaffian kernel}.
	\begin{lem}
		\label{lem:error estimate for Pfaffian kernel}
		Let us fix $\tau\in [0,1)$ and denote
		\[R_N(x,y):=K_N(x,y)-\cal{K}_{\tau}(\sq{N}(x-y)).\]
		Then for any $x,y\in (-1-\tau,1+\tau)$ there is $C,c>0$ such that
		\[\|R_N(x,y)\|_F\le Ce^{-Nc},\]
		where $\|A\|_F^2=\Tr(AA^T)$ is the Frobenius norm.\\
		Moreover, for any $\epsilon>0$ and $k\ge 0$ we may choose $c$ small enough that for $x,y\in (-1-\tau+\epsilon,1+\tau-\epsilon)$ we have that
		\[\sum_{0\le i,j\le k}\|\partial_x^i\partial_y^jR_N(x,y)\|_F\le Ce^{-cN}.\]
	\end{lem}
	
	We note that it has been known that $K_N$ converges $\cal{K}_{\tau}$ in various regimes. For example, when $N$ is even and $\tau \neq 0$, the proof of (5.25) in \cite{forrester-elliptic} easily generalizes to show this local convergence in the bulk. Hence the content of Lemma \ref{lem:error estimate for Pfaffian kernel} is to establish the rate and uniformity. This is important as such a level of precision allows us to replace $K_N$ in the expression for the determinant given by Lemma \ref{lem:tuca-relation} at the level required to show Theorem \ref{theorem: moments of real bulk}, as we show in the following corollary.
	
	\begin{corr}
		\label{corr:error estimate for pfaffian}
		Fix $\ell\ge 1$, $\epsilon>0$ and $K\ge 0$, and define
		\[\bar{R}_N(\mu_1,\cdots,\mu_\ell):=\pf(\l[K_{N}\l(\mu_i,\mu_j\r)\r]_{1\le i,j\le \ell})-\pf\l(\l[\cal{K}_{\tau}\left(\sqrt{N}\l(\mu_i-\mu_j\r)\right)\r]_{1\le i,j\le \ell}\r).\]
		Then there is $C,c>0$ such that for $\mu_1,\cdots,\mu_\ell\in (-(1+\tau)+\epsilon,1+\tau-\epsilon)$ we have
		\[\sum_{0\le i_1,\cdots,i_\ell\le K}|\partial_{\mu_1}^{i_1}\cdots \partial_{\mu_\ell}^{i_l}\bar{R}_N(\mu_1,\cdots,\mu_\ell)|\le Ce^{-cN}.\]
	\end{corr}
	\begin{proof}
		Employing the Leibniz formula for the Pfaffian, and the product rule for differentials, we may write $\partial_{\mu_1}^{i_1}\cdots \partial_{\mu_\ell}^{i_l}\pf([K_N(\mu_i,\mu_j)]_{1\le i,j\le \ell})$ as a polynomial of degree $\ell$ in the entries of $\{\partial_{\mu_i}^{k}\partial_{\mu_j}^{l}K_N(\mu_i,\mu_j)\}_{1\le i,j\le \ell, 1\le k,l\le K}$, and similarly for $\partial_{\mu_1}^{i_1}\cdots \partial_{\mu_\ell}^{i_l}\pf([\cal{K}_{\tau}(\sqrt{N}(\mu_i-\mu_j))]_{1\le i,j\le \ell})$. Letting $R_N$ be as in Proposition \ref{prop:main prop}, we may use the equality $K_N(x,y)=\cal{K}_{\tau}(\sqrt{N}(x-y))+R_N(x,y)$ to expand each term in this polynomial, and conclude that for some large constant $C_0:=C_0(\ell,K)$ we have the course bound
		\[\sum_{0\le i_1,\cdots,i_\ell\le K}\left|\partial_{\mu_1}^{i_1}\cdots \partial_{\mu_\ell}^{i_l}\bar{R}_N(\mu_1,\cdots,\mu_\ell)\right|\le C_0\sum_{s=1}^{\ell}\sum_{i,j=1}^{\ell}E_{N,s}(\mu_i,\mu_j)\]
		where here
		\[E_{N,s}(\mu',\mu'')=\left(\sum_{i,j=0}^{K}\|\partial_{\mu'}^i\partial_{\mu''}^jR_N(\mu',\mu'')\|_F\right)^{s}\left(\sum_{i,j=0}^{K}\|\partial_{\mu'}^i\partial_{\mu''}^j\cal{K}_{\tau}(\sqrt{N}(\mu'-\mu''))\|_F\right)^{\ell-s}.\]
		Inspection of $\cal{K}(\sqrt{N}(\mu'-\mu''))$ shows there is $C$ such that for any $\mu',\mu''\in \R$,
		\[\|\partial_{\mu'}^i\partial_{\mu''}^j\cal{K}_{\tau}(\sqrt{N}(\mu'-\mu''))\|_F\le C N^{i+j+1/2}.\]
		Moreover, we see from Lemma \ref{lem:error estimate for Pfaffian kernel} there is $C,c>0$ such that
		\[\sum_{i,j=0}^{K}\|\partial_{\mu'}^i\partial_{\mu''}^jR_N(\mu',\mu'')\|_F=Ce^{-Nc}.\]
		Thus shrinking $c>0$ and enlarging $C>0$ completes the proof.
	\end{proof}
	
	By Lemmas \ref{lem:tuca-relation} and \ref{lem:limit for D} and (\ref{eqn:even Pfaffian formula}) we see that for any distinct $\mu_1,\cdots,\mu_\ell\in \R$, we have that
	\[\E\left[\prod_{i=1}^{\ell}|\det(A_N-\mu_iI)|\right]=\]\[e^{\frac{N \ell}{2} \left(\frac{\mu^2}{1+\tau}-1\right)}\left(4\pi (1+\tau)\right)^{\ell/2}\frac{\pf(\l[K_{N+\ell}\l(\tilde{\mu}_i,\tilde{\mu}_j\r)\r]_{1\le i,j\le \ell})}{|\Delta(\mu_1,\dots \mu_\ell)|}\l(1+O\left(N^{-1}\right)\r).\label{eqn:determinant in terms of pfaffian}\]
	where the error term is independent of $(\mu_1,\cdots,\mu_\ell)$. Thus, we primarily need to better understand limiting expressions in Corollary \ref{corr:error estimate for pfaffian}. To proceed, we must understand $\cal{K}_\tau$ better. The key is to note that for any $\tau\in [0,1)$, we have that
	\[\begin{bmatrix}1 & 0\\ 0 & \sqrt{1-\tau^2} \end{bmatrix}\cal{K}_{\tau}\left(x\sqrt{1-\tau^2}\right)\begin{bmatrix}1 & 0\\ 0 & \sqrt{1-\tau^2} \end{bmatrix}=\cal{K}_{0}(x)=\cal{H}(x),\]
	where $\cal{H}$ is the kernel in Theorem \ref{theorem: real bulk multi point}. In particular, for any $\mu_1,\cdots,\mu_\ell\in \R$ we have
	\[\pf(\l[\cal{K}_{\tau}\l(\mu_i-\mu_j\r)\r]_{1\le i,j\le \ell})=\frac{1}{(1-\tau^2)^{\ell/2}}\pf\l(\left[\cal{H}\left(\frac{\mu_i-\mu_j}{\sqrt{1-\tau^2}}\right)\r]_{1\le i,j\le \ell}\r).\label{eqn:Kt and H pfaffian relation}\]
	
	\begin{proof}[Proof of Theorem \ref{theorem: real bulk multi point}]
		Noting that $\Delta(\mu_1,\cdots,\mu_\ell)$ is a homogeneous polynomial of degree $\ell(\ell-1)/2$, applying Corollary \ref{corr:error estimate for pfaffian} and (\ref{eqn:Kt and H pfaffian relation}) to (\ref{eqn:determinant in terms of pfaffian}), and noting the relation $\sqrt{N+\ell}\tilde{\mu}_i=\sqrt{N}\mu_i$ completes the proof.
	\end{proof}
	
	With Theorem \ref{theorem: real bulk multi point} proven, we now address the case of higher moments, in which we want to take multiple points $\mu_1,\cdots,\mu_\ell$ in (\ref{eqn:determinant in terms of pfaffian}) to the same point $\mu$. The left-hand side is obviously continuous in $(\mu_1,\cdots, \mu_\ell)$. However, the ratio on the right-hand side has the Vandermonde determinant in the denominator, which must thus be canceled by a similar factor in $\pf(\l[K_{N+\ell}\l(\tilde{\mu}_i,\tilde{\mu}_j\r)\r]_{1\le i,j\le \ell})$ as we know the result must be continuous. 
	
	Computing the result however will require some terminology. For this, we note that for any $k$-variate polynomial 
	\[f(\mu_1,\cdots,\mu_k)=\sum_{i_1,\cdots,i_k=1}^{\infty}a_{i_1,\cdots,i_k}\mu^{i_1}\cdots \mu^{i_k},\]
	we may associate a differential operator on $\R^\ell$ by
	\[\partial_f:=\sum_{i_1,\cdots,i_k=1}^{\infty}a_{i_1,\cdots,i_k}\frac{\partial^{\sum_{l=1}^k i_l}}{\partial^{i_1} \mu_1\cdots \partial^{i_k} \mu_k}.\]
	It is clear that this map preserves both addition and multiplication of polynomials.
	
	We are interested in applying this to the Vandermonde determinants. For clarity here, we will momentarily denote the Vandermonde determinant in $k$-variables as $\Delta_k$ and denote the associated differential operator $\partial_{\Delta_k}$. We note this is a homogeneous differential operator of order $k(k-1)/2$. Next recall by the Leibniz formula we have that
	\[\Delta_k(\mu_1,\cdots, \mu_k)=\sum_{\sigma\in \Sigma_k}(-1)^{\sigma}\mu^{\sigma(1)-1}_1\cdots \mu^{\sigma(k)-1}_k,\]
	where $\Sigma_k$ is the symmetric group on $k$-elements, and $(-1)^\sigma$ denotes the parity of a permutation $\sigma\in \Sigma_k$. In particular, we have that
	\[\partial_{\Delta_k}=\sum_{\sigma\in \Sigma_k}(-1)^{\sigma}\frac{\partial^{k(k-1)/2}}{\partial^{\sigma(1)-1}  \mu_{1}\partial^{\sigma(2)-1}  \mu_{2}\cdots \partial^{\sigma(k)-1}  \mu_{k}}.\]
	From these expression, it is clear that when computing $\partial_{\Delta_k}\Delta_k$, each partial derivative term vanishes on all terms but its corresponding multinomial term, so that
	\[\partial_{\Delta_k}\Delta_k=\sum_{\sigma\in \Sigma_k}(-1)^{2\sigma}\frac{\partial^{k(k-1)/2}\mu^{\sigma(1)-1}_1\cdots \mu^{\sigma(k)-1}_k}{\partial^{\sigma(1)-1}  \mu_{1}\partial^{\sigma(2)-1}  \mu_{2}\cdots \partial^{\sigma(k)-1}  \mu_{k}}=k!\prod_{i=1}^{k}(k-1)!=\prod_{i=1}^{k+1}\Gamma(i).\label{eqn:Delta_k constant}\]
	
	We now give a purely algebraic result concerning the limit after one divides by the Vandermonde determinant.
	
	\begin{lem}
		\label{lem:abstract-nonsense complex analysis}
		Fix $\ell\in \N$, let $U\subseteq \C^\ell$ be an open set, and let $f:U\to \C$ be a holomorphic function such that $f(\mu_1,\cdots, \mu_\ell)=0$ if $\mu_j=\mu_k$ for any $j\neq k$. Then for any $\mu\in \C$ such that $(\mu,\cdots,\mu)\in U$ we have that
		\[\lim_{\mu_1,\cdots ,\mu_\ell\to \mu}\frac{f(\mu_1,\cdots,\mu_\ell)}{\Delta(\mu_1,\cdots,\mu_\ell)}=\frac{\partial_{\Delta_\ell}f(\mu,\cdots,\mu)}{\prod_{i=1}^{\ell+1}\Gamma(i)}.\]
	\end{lem}
	\begin{proof}
		We first show that there exists a unique holomorphic function $g:U\to \C$ such that
		\[g(\mu_1,\cdots,\mu_\ell)=\frac{f(\mu_1,\cdots,\mu_\ell)}{\Delta(\mu_1,\cdots,\mu_\ell)}.\label{eqn:division by Vandermonde}\]
		For this, note that as $\Delta_\ell$ is the minimal polynomial corresponding the to union of hyperplanes
		\[\C^{\ell}_{=}:=\{(\mu_1,\cdots, \mu_\ell)\in \C^\ell: \mu_i=\mu_j \text{ for some }i\neq j\},\]
		it follows from the analytic Nullstellensatz (see Proposition 1.1.29 of \cite{complex-geo}) that such $g$ uniquely exists in a neighborhood of any given point of $U$. However, as $g$ is clearly uniquely determined on $U\setminus \C^{\ell}_{=}$, these must glue to a globally defined function on $U$.
		
		Now we must only verify that
		\[\left(\partial_{\Delta_\ell}\Delta_\ell\right)g(\mu,\cdots,\mu)=\partial_{\Delta_\ell}f(\mu_1,\cdots,\mu_\ell).\]
		For this, note that
		\[\partial_{\Delta_\ell}f=\partial_{\Delta_\ell}\left(g\Delta_{\ell}\right)=\prod_{i<j}\left(\partial_{\mu_i}-\partial_{\mu_j}\right)\left(g\Delta_{\ell}\right).\label{eqn:ignore-23893827}\]
		However, for any non-zero subset $S\subseteq \{(i,j)\in \N: i<j\}$, it is clear that
		\[\prod_{i<j,(i,j)\notin S}\left(\partial_{\mu_i}-\partial_{\mu_j}\right)\Delta_{\ell}(\mu_1,\cdots,\mu_\ell)=\prod_{i<j,(i,j)\notin S}\left(\partial_{\mu_i}-\partial_{\mu_j}\right)\left(\prod_{i<j}\left(\mu_i-\mu_j\right)\right),\]
		which must vanish when $\mu_1=\cdots=\mu_\ell$ by employing the product rule repeatedly on the right. So repeatedly applying the product rule again to (\ref{eqn:ignore-23893827}), we see that
		\[\partial_{\Delta_{\ell}}f=\partial_{\Delta_\ell} (g\Delta_\ell)=g\partial_{\Delta_\ell}\Delta_\ell+h\]
		where the function $h$ which vanishes when $\mu_1=\cdots =\mu_\ell$. This immediately gives the desired result.
	\end{proof}
	
	We may now use this lemma to get an expression for the right-hand expression in (\ref{eqn:determinant in terms of pfaffian}) when the points are limited to the same point. For this, we will need to take care to consider the ordering of our points, as Lemma \ref{lem:tuca-relation} involves division by the absolute value of the Vandermonde determinant, not the Vandermonde determinant itself.
	
	For this it is useful to introduce the subset $\R^\ell_{\ge }\subseteq \R^{\ell}$ of ordered sequences given by
	\[\R^\ell_{\ge }=\{(\mu_1,\cdots, \mu_\ell)\in \R^\ell:\mu_1\ge \mu_2\ge \cdots\ge \mu_\ell\},\]
	and similarly for $\R^{\ell}_{>}$. When a function $f:\R^\ell_{>}\to \R$ is such that $\partial_{\Delta_\ell}f$ exists everywhere, and extends continuously to $\R^\ell_{\ge }$, we will denote the resulting extension as $\partial_{\Delta_\ell}^+f$. When $f$ extends to a smooth function $\bar{f}$ on a neighborhood of $\R_{\ge}^{\ell}\subset \R^\ell$, $\partial_{\Delta_\ell}^+f=\partial_{\Delta_\ell}\bar{f}$. However, this restricted differential is convenient, as it also satisfies $\partial_{\Delta_\ell}^+|\Delta_\ell|=\partial_{\Delta_\ell}\Delta_\ell$.
	
	\begin{lem}
		\label{lem:ratio limit for Kn}
		For any $x\in (-1-\tau,1+\tau)$, and either $K=K_N$ or $K=\cal{K}_\tau$, we have that
		\[\lim_{\mu_1>\mu_2>\cdots>\mu_\ell\to \mu}\frac{\pf([K(\mu_i,\mu_j)]_{1\le i,j\le \ell})}{|\Delta(\mu_1,,\mu_\ell)|}=\frac{1}{\prod_{i=1}^{\ell+1}\Gamma(i)}\partial_{\Delta_\ell}^+\pf([K(\mu_i,\mu_j)]_{1\le i,j\le \ell})\big|_{\mu_1=\mu_2=\cdots=\mu_\ell=\mu},\]
		where we take $\cal{K}_\tau(x,y)$ to mean $\cal{K}_{\tau}(x,y)=\cal{K}_{\tau}(x-y)$.
	\end{lem}
	
	\begin{proof}
		We observe that it is clear that $\pf([K(\mu_i,\mu_j)]_{1\le i,j\le \ell})=0$ if any points coincide.
		The limit in question is taken over a subset of $\R^{\ell}_{> }$, where $|\Delta(\mu_1,\cdots,\mu_\ell)|=\Delta(\mu_1,\cdots,\mu_\ell)$. Thus our result would follow from Lemma \ref{lem:abstract-nonsense complex analysis} if we are able to show that the function \[f(\mu_1,\cdots,\mu_\ell)=\pf([K(\mu_i,\mu_j)]_{1\le i,j\le \ell}),\]
		defined on the interior of $\R^\ell_{>}$ extends to an analytic function on $\C^\ell$. We will use the notation $S,D,I$ for the functions corresponding to our choice of $K$. One may check in either case that $S$, and thus $D$, clearly has such an extension. However, the same cannot be said for $K(x,y)$ due only to the term $\frac{1}{2}\sign(y-x)$ in $I(x,y)$. 
		
		However, over the subset $\R^\ell_{>}$, this factor is constant in each matrix element. In particular, if we define
		\[\hat{K}_{ij}(x,y):=\begin{bmatrix}
			-\hat{I}_{ij}(x,y)& S(x,y)\\
			-S(y,x)& D(x,y)
		\end{bmatrix},\label{eqn:ignore-352}\]
		\[\hat{I}_{ij}(x,y)=\frac{1}{2}\sign(i-j)-\int_x^y S(x,z)dz,\]
		then for $(\mu_1,\cdots,\mu_\ell)\in \R^\ell_{>}$ we have that
		\[f(\mu_1,\cdots,\mu_\ell)=\pf([\hat{K}_{ij}(\mu_i,\mu_j)]_{1\le i,j\le \ell}).\]
		Moreover, it is clear that each $\hat{K}_{ij}$ extends to an analytic function on $\C^2$, and thus so does $f$, completing the proof.
	\end{proof}	
	
	From this we can now obtain the following result.
	
	\begin{prop}
		\label{prop:main prop}
		For any $\ell\in \N$ and $x\in (-(1+\tau),1+\tau)$, then we have that
		\[\E[|\det(A_N-\mu I)|^{\ell}]=\left(1+O(N^{-1})\right)e^{\frac{N\ell}{2} \left(\frac{\mu^2}{1+\tau}-1\right)}N^{\ell(\ell-1)/4}\frac{ (1+\tau)^{\ell/2}}{(1-\tau^2)^{\ell(\ell+1)/4}}\cal{C}_{\ell}\]
		where here $\cal{C}_{\ell}$ is the positive constant given by
		\[\cal{C}_{\ell}:=\frac{\left(4\pi\right)^{\ell/2}}{\prod_{i=1}^{\ell+1}\Gamma(i)}\partial_{\Delta_\ell}^+\pf([\cal{H}(\mu_i-\mu_j)]_{1\le i,j\le \ell})\big|_{\mu_1=\cdots=\mu_\ell=0}.\]
		Moreover, this bound is compact-uniform in the choice of $x\in (-1-\tau,1+\tau)$.
	\end{prop}
	\begin{proof}
		By noting that $\partial_{\Delta_\ell}$ is a homogeneous differential operator of order $\ell(\ell-1)/2$, and employing translation invariance, we see from (\ref{eqn:Kt and H pfaffian relation}) that
		\[
		\partial_{\Delta_\ell}^+\cal{K}_\tau(\sqrt{N}(\mu_i-\mu_j))\big|_{\mu_1=\cdots=\mu_\ell=\mu}=\frac{N^{\ell(\ell-1)/4}}{(1-\tau^2)^{\ell(\ell-1)/4}}\partial_{\Delta_\ell}^+\pf([\cal{H}(\mu_i-\mu_j)]_{1\le i,j\le \ell})\big|_{\mu_1=\cdots=\mu_\ell=0}.
		\]
		Thus applying this, Lemma \ref{lem:ratio limit for Kn} and Corollary \ref{corr:error estimate for pfaffian} to the expression (\ref{eqn:determinant in terms of pfaffian}), we see that there is small $c>0$ such that
		\[
		\E\left[\left|\det(A_{N}-\mu I)\right|^{\ell}\right]=e^{\frac{N\ell}{2} \left(\frac{\mu^2}{1+\tau}-1\right)}\left((1+O\left(N^{-1}\right))N^{\ell(\ell-1)/4}\frac{ (1+\tau)^{\ell/2}}{(1-\tau^2)^{\ell(\ell+1)/4}}\cal{C}_{\ell}+O\left(e^{-Nc}\right)\right).\label{eqn:ignore-len}
		\]
		Thus the desired statement would follow if we show that $\cal{C}_\ell$ is indeed positive. For this, we note that by (\ref{eqn:ignore-len}), and positivity of the left-hand side, we must have that $\cal{C}_\ell\ge 0$. Moreover, using the lower bound
		\[\E[|\det(A_N-\mu I)|^{\ell}]\ge \E[|\det(A_N-\mu I)|]^{\ell},\]
		we may apply (\ref{eqn:ignore-len}) to both sides, which easily shows that if $\cal{C}_\ell=0$ then $\cal{C}_1=0$. Thus we only need to check that $\cal{C}_1\neq 0$. However in this case
		\[\cal{C}_1=\frac{\left(4\pi\right)^{1/2}}{\Gamma(1)\Gamma(2)}\pf(\cal{H}(0))=\frac{(4\pi)^{1/2}}{\sqrt{2\pi}}=\sqrt{\pi}>0,\]
		which completes the proof.
	\end{proof}
	
	With Proposition \ref{prop:main prop}, Theorem \ref{theorem: moments of real bulk} follows from the following result.
	
	\begin{lem}
		\label{lem:identification of constant}
		For any $\ell\in \N$, we have that
		\[\cal{C}_{\ell}=\frac{(2\pi)^{\ell/2}}{2^{\ell(\ell-1)/4}\prod_{j=1}^{\ell}\Gamma(j/2)}.\]
	\end{lem}
	
	We note, given the combinatorial nature $\cal{C}_{\ell}$'s definition, we suspect that this identity has a direct proof. However, we have only been able to directly verify this result when $\ell$ is even, where the constant may be related to a certain matrix-integral, as discussed above. On the other hand, we are able to pursue an alternative route to show this by comparing this constant to the one appearing in the known computation of the absolute determinant of the real Ginibre ensemble. 
	
	\begin{proof}[Proof of Lemma \ref{lem:identification of constant}]
		
		It is known that the absolute determinant of the Ginibre ensemble may be expressed as a ratio of Selberg integrals. This gives exact formula (see for example 3.1 of \cite{detginibre})
		\[\E|\det(A_N)|^\ell=N^{-N\ell/2}2^{N\ell/2}\prod_{i=1}^{N}\frac{\Gamma((i+\ell)/2)}{\Gamma(i/2)}=N^{-N\ell/2}2^{N\ell/2}\frac{\prod_{i=N+1}^{N+\ell}\Gamma(i/2)}{\prod_{i=1}^{\ell}\Gamma(i/2)}.\label{eqn:determinant of ginibre ratio}\]
		Employing Stirling's approximation as in Lemma \ref{lem:limit for D} we see
		\[\prod_{i=N+1}^{N+\ell}\Gamma(i/2)=\left(1+O(N^{-1})\right)(2\pi)^{\ell/2}2^{-(\ell-1)\ell/4}(2e)^{-N\ell/2}N^{N\ell/2+\ell(\ell-1)/4}.\]
		In particular, we have that
		\[\E|\det(A_N)|^\ell=\left(1+O(N^{-1})\right)e^{-N\ell/2}N^{\ell(\ell-1)/4}\left(\frac{(2\pi)^{\ell/2}}{2^{\ell(\ell-1)/4}\prod_{i=1}^{\ell}\Gamma(i/2)}\right).\]
		Comparing this with the expression in Proposition \ref{prop:main prop} completes the proof.
	\end{proof}
	
	\section{The Pfaffian Kernel and The Proof of Lemma \ref{lem:error estimate for Pfaffian kernel} \label{section:pfaffian kernel}}
	
	What remains is to demonstrate Lemma \ref{lem:error estimate for Pfaffian kernel}. However, before we can proceed to this, we will of course need to give an explicit form for the kernel $K_N(x,y)$. Recalling (\ref{eqn:Pfaffian K-def}) above, all that is left to give the kernel is to give the function $S_N$. 
	
	In the case of $\tau\neq 0$ and $N$ even, a formula for this is given in \cite{forrester-elliptic}, which we recall. When $\tau\neq 0$ and $N$ is odd however, we need to translate the results of \cite{forrester-elliptic} into the method developed by \cite{sinclair}, which is somewhat lengthy. Luckily when $\tau=0$, these formulas are well-known \cite{borodin-elliptic,ginibre-odd}. The general formula is given in the following lemma.
	
	\begin{lem}
		\label{lem:exact form of pfaffian kernel}
		Let us assume that $\tau\in (0,1)$. Then for any $N\ge \ell$ the formulas (\ref{eqn:even Pfaffian formula}), (\ref{eqn:Pfaffian K-def}) and  (\ref{eqn:Pfaffian I and D def}) hold with $S_N$ given by
		\[S_N\left(\frac{x}{\sqrt{N}},\frac{y}{\sqrt{N}}\right)=\frac{e^{-\frac{x^2+y^2}{2(1+\tau)}}}{\sq{2\pi}}\sum_{k=0}^{N-2}\frac{1}{k!}C_k(x)C_k(y)+\frac{e^{-\frac{y^2}{2(1+\tau)}}}{2\sq{2\pi}(1+\tau)}\frac{C_{N-1}(y)\phi_{N-2}(x)}{(N-2)!},\label{eqn:S-definition}\]
		where here
		\[C_k(x):=
		\begin{cases}
			\left(\frac{\tau}{2}\right)^{k/2}H_k \left(\frac{x}{\sq{2\tau}}\right)\;\;\; \tau\neq 0\\
			x^k\;\;\; \tau=0
		\end{cases}	
		,\]
		$H_k$ denotes the $k$-th Hermite polynomial, and
		\[\phi_N(x):=\left(\frac{\sq{2\pi(1+\tau)}N!}{2^{N/2}\Gamma(N/2+1)}\right)-2\int_{x}^\infty e^{-\frac{t^2}{2(1+\tau)}}C_N(t)dt.\label{eqn:ignore-phi-def}\]
	\end{lem}
	\begin{proof}
		It is convenient to define
		\[\bar{S}_N\left(x,y\right)=S_N\left(\frac{x}{\sqrt{N}},\frac{y}{\sqrt{N}}\right),\]
		and similarly for $\bar{K}_N$, $\bar{D}_N$, and etc. We now proceed to the first case.
		
		\textbf{Case: $\tau\neq 0$ and $N$ even } 
		This case will follow almost immediately from \cite{forrester-elliptic}. Referencing \cite{forrester-elliptic} repeating, the expression for the correlation function (after rescaling) in terms of $K_N$ is their (4.3), our (\ref{eqn:Pfaffian K-def}) and (\ref{eqn:Pfaffian I and D def}) is their (4.4) and (4.6), and finally their definition of $\Sb_N$ and $\phi_N$ is their (5.24) and (5.16), as we have that our $\phi_N$ is their $\Phi_N$ when $N$ is even.
		
		\textbf{Case: $\tau\neq 0$ and $N$ odd} This will consist of employing the formulas of \cite{forrester-elliptic}, to the set-up used to treat odd correlation functions in \cite{sinclair}. In particular, we see that the eigenvalue process for $\sqrt{N}GEE(\tau,N)$ is of the form considered in (1.4) of \cite{sinclair} with
		\[\bar{w}(z)=e^{-\frac{z^2}{2(1+\tau)}}\sq{\erfc \left(|\Im(z)|\sq{\frac{2}{1-\tau^2}} \right)}.\]
		In particular, our expression for $\bar{K}_N$ will follow from the formulas they provide in their Section 7, using the family of monic skew-orthogonal polynomials found in Theorem 1 of \cite{forrester-elliptic}. However, due to a number of differences in notation between \cite{forrester-elliptic} and \cite{sinclair}, we will provide some expressions for how the functions defined in \cite{forrester-elliptic} equate to the functions defined in \cite{sinclair}. 
		
		In particular, by Theorem 1 of \cite{forrester-elliptic} we may apply Section 7 of \cite{sinclair} with
		\[\tilde{q}_{2i}(x)=e^{-\frac{x^2}{2(1+\tau)}}C_{2i}(x),\;\;\; \tilde{q}_{2i+1}(x)=e^{-\frac{x^2}{2(1+\tau)}}(C_{2i+1}(x)-2iC_{2i-1}(x)).\]
		\[2\tilde{\epsilon}q_{i}(x)=\int \sign(x-t)\tilde{q}_i(t)dt=:\Phi_i(x),\label{eqn:ignore-luna}\]
		where this $\Phi_i$ coincides with the function defined in \cite{forrester-elliptic}. Employing their Theorem 1 of \cite{forrester-elliptic} again and the formula (see \cite{AS})
		\[\int e^{-x^2}H_{2m}(xy)dx=2^{2m}\Gamma\left(m+1/2\right)(y^2-1)^m,\]
		we see that the constants in the expression for $\Sb_N$ in \cite{sinclair} (defined on pg. 29) are given by
		\[r_i=-2\sq{2\pi}(1+\tau)(2i)!,\;\; s_{2i+1}=0, \;\; s_{2i}=\sq{2(1+\tau)}2^{i}\Gamma\left(i+1/2\right).\]
		To avoid confusion, there is a sign difference in the definition of $r_i$ in \cite{sinclair} and \cite{forrester-elliptic}, which occurs due to the sign difference between their skew-symmetric inner products. With these relations, the first expression on pg. 31 of \cite{sinclair} reads 
		\footnote{Note as $s_{2j+1}=0$, terms in both the second and third sum in \cite{sinclair} vanish, which we have grouped into the first sum of $II$.}
		\[\Sb_N(y,x)=I+II,\]
		\[I=\sum_{j=0}^{(N-3)/2}\frac{2}{r_j}\left(\tilde{q}_{2j}(y)\tilde{\epsilon}q_{2j+1}(x)-\tilde{q}_{2j+1}(y)\tilde{\epsilon}q_{2j}(x)\right),\]
		\[II=\sum_{j=0}^{(N-3)/2}\frac{2 s_{2j}}{r_j s_{N-1}}\left(\tilde{q}_{2j+1}(y)\tilde{\epsilon}q_{N-1}(x)-\tilde{q}_{N-1}(y)\tilde{\epsilon}q_{2j+1}(x)\right)+\frac{\tilde{q}_{N-1}(y)}{s_{N-1}}.\]
		Note this formula is for $\Sb_N(y,x)$ and indeed, our first step will be to show that (\ref{eqn:S-definition}) holds for $\Sb_N(y,x)$ not $\Sb_N(x,y)$. We will return below to this point when we relate the form of $\bar{K}_N$ given in \cite{sinclair} to our $\bar{K}_N$ above.
		
		Before this though, we simplify $\Sb_N(y,x)$. We note that with the above relations, $I$ coincides with (4.5) of \cite{forrester-elliptic} after making the replacement $N\mapsto N-1$. That is, $I=\bar{S}_{N-1}(y,x)$ except in the form \cite{forrester-elliptic} first introduces it, and so is treated by the previous case.
		
		Thus we are left to study $II$. For this, we note by the Legendre duplication formula that
		\[\frac{2s_{2i}}{s_{N-1}r_i}=\frac{-1}{2^{N/2}\Gamma(N/2)(1+\tau)}\frac{2^{i+1/2}\Gamma(j+1/2)}{2\sqrt{2\pi}\Gamma(2i+1)}=\frac{-1}{2^{N/2}\Gamma(N/2)(1+\tau)}\frac{1}{2^{i}{i!}}.\]
		We may thus evaluate the first term in the summation in $II$ as
		\[\sum_{j=0}^{(N-3)/2}\frac{2s_{2j}}{s_{N-1}r_j}\tilde{\epsilon}q_{N-1}(x)\tilde{q}_{2j+1}(y)=\]
		\[-\frac{\Phi_{N-1}(x)}{2^{N/2}\Gamma(N/2)(1+\tau)}\sum_{j=0}^{(N-3)/2}\frac{e^{-\frac{y^2}{2(1+\tau)}}}{2}\left(\frac{1}{j!2^j}C_{2j+1}(y)-\frac{j}{j!2^{j-1}}C_{2j-1}(y)\right)=\]
		\[-\frac{e^{-\frac{y^2}{2(1+\tau)}}}{2^{N/2}\Gamma(N/2)(1+\tau)}\frac{\Phi_{N-1}(x)C_{N-2}(y)}{((N-3)/2)!2^{(N-1)/2}}=-\frac{e^{-\frac{y^2}{2(1+\tau)}}C_{N-2}(y)\Phi_{N-1}(x)}{2\sqrt{2\pi}(1+\tau)(N-2)!},\]
		where in the last step we have again used Legendre's duplication formula to conclude that
		\[
		2^{N-2}\Gamma((N-1)/2)\Gamma(N/2)=\sqrt{\pi}\Gamma(N-1)=\sqrt{\pi}(N-2)!.
		\]
		Similarly we may evaluate the second term in the sum as
		\[-\sum_{j=0}^{(N-3)/2}\frac{2s_{2j}}{s_{N-1}r_j}\tilde{q}_{N-1}(y)\tilde{\epsilon}q_{2j+1}(x)=\]
		\[\frac{e^{-\frac{y^2}{2(1+\tau)}}C_{N-1}(y)}{2^{N/2}\Gamma(N/2)(1+\tau)}\sum_{j=0}^{(N-3)/2}\left(\frac{\Phi_{2j+1}(t)}{j!2^j}-\frac{j\Phi_{2j-1}(t)}{j!2^{j-1}}\right)=\]
		\[\frac{e^{-\frac{y^2}{2(1+\tau)}}C_{N-1}(y)}{2^{N/2}\Gamma(N/2)(1+\tau)}\frac{\Phi_{N-2}(t)}{((N-3)/2)!2^{(N-1)/2}}=\frac{e^{-\frac{y^2}{2(1+\tau)}}C_{N-1}(y)\Phi_{N-2}(x)}{2\sqrt{2\pi}(1+\tau)(N-2)!}.\]
		Finally noting that	\[\frac{\tilde{q}_{N-1}(y)}{s_{N-1}}=\frac{e^{-\frac{y^2}{2(1+\tau)}}C_{N-1}(y)}{2^{N/2}\Gamma(N/2)\sqrt{1+\tau}},\]
		we see that 
		\[II=-\frac{e^{-\frac{y^2}{2(1+\tau)}}C_{N-2}(y)\Phi_{N-1}(x)}{2\sqrt{2\pi}(N-2)!}+\frac{e^{-\frac{y^2}{2(1+\tau)}}C_{N-1}(y)\Phi_{N-2}(x)}{2\sqrt{2\pi}(1+\tau)(N-2)!}+\frac{e^{-\frac{y^2}{2(1+\tau)}}C_{N-1}(y)}{2^{N/2}\Gamma(N/2)\sqrt{1+\tau}}.\]
		We also note that by (4.40) in \cite{forrester-elliptic}, we have that
		\[\Phi_{N-1}(x)-(N-2)\Phi_{N-3}(x)=-2(1+\tau)e^{-\frac{x^2}{2(1+\tau)}}C_{N-2}(x),\]
		so that in particular, when combined with the last term of $\bar{S}_{N-1}(y,x)$, one gets
		\[\frac{e^{-\frac{y^2}{2(1+\tau)}}C_{N-2}(y)\Phi_{N-3}(x)}{2\sq{2\pi}(1+\tau)(N-3)!}-\frac{e^{-\frac{y^2}{2(1+\tau)}}C_{N-2}(y)\Phi_{N-1}(x)}{2\sqrt{2\pi}(1+\tau)(N-2)!}=\frac{e^{-\frac{x^2+y^2}{2(1+\tau)}}C_{N-2}(x)C_{N-2}(y)}{\sqrt{2\pi}(N-2)!}.\]
		Moreover, when $N$ is odd we have that
		\[\Phi_{N}(x)=-2\int_{x}^\infty e^{-\frac{t^2}{2(1+\tau)}}C_N(t)dt,\]
		so that the last two terms of $II$ satisfy
		\[\frac{e^{-\frac{y^2}{2(1+\tau)}}C_{N-1}(y)\Phi_{N-2}(x)}{2\sqrt{2\pi}(1+\tau)(N-2)!}+\frac{e^{-\frac{y^2}{2(1+\tau)}}C_{N-1}(y)}{2^{N/2}\Gamma(N/2)\sqrt{1+\tau}}=\frac{e^{-\frac{y^2}{2(1+\tau)}}}{2\sq{2\pi}(1+\tau)}\frac{C_{N-1}(y)\phi_{N-2}(x)}{(N-2)!}.\]
		Together with the relation $I=\bar{S}_{N-1}(y,x)$, this establishes (\ref{eqn:S-definition}) with $\bar{S}_N(y,x)$ instead of $\bar{S}_N(x,y)$ on the left hand side. 
		
		We now deal with relating their $\bar{K}_N$ to ours. On pg. 27 of \cite{sinclair}, they give an expression for $\bar{K}_N(y,x)$ equivalent to 
		\[\bar{K}_N(y,x)=\begin{bmatrix}
			D\Sb_N(y,x)& \Sb_N(y,x)\\
			-\Sb_N(x,y)& \Sb_NI(y,x)+\frac{1}{2}\sign(x-y)\\
		\end{bmatrix}.\label{eqn:odd-temp-2}\]
		We note that it is clear that $D\Sb_N(y,x)=\partial_x \Sb_N(y,x)$ and $\Sb_N(y,x)=\partial_y \Sb_NI(y,x)$. In particular, the difference $\Sb_NI(y,x)-\int_x^y \Sb_N(z,x)dz$ is a function of only $x$, as its partial derivative in $y$ is zero, but as they both vanish on $y=x$ (note $\Sb_NI$ and $D\Sb_N$ are anti-symmetric), we have that $\Sb_NI(y,x)=\int_x^y \Sb_N(z,x)dz$, as desired. This shows the matrix elements satisfy the stated relations (\ref{eqn:Pfaffian I and D def}).
		
		Finally, we deal with the difference of the forms of $\bar{K}_N$ in (\ref{eqn:odd-temp-2}) and (\ref{eqn:Pfaffian K-def}), as well as the the switching of the order of $x$ and $y$, whose effects cancel each other out. We note that
		\[-\begin{bmatrix}0 & 1\\ 1 & 0 \end{bmatrix}\begin{bmatrix}
			D\Sb_N(y,x)& \Sb_N(y,x)\\
			-\Sb_N(y,x)& \Sb_NI(y,x)+\frac{1}{2}\sign(x-y)\\
		\end{bmatrix}\begin{bmatrix}0 & 1\\ 1 & 0 \end{bmatrix}=\]
		\[\begin{bmatrix}
			\frac{1}{2}\sign(y-x)+\Sb_NI(x,y)& \Sb_N(x,y)\\
			-\Sb_N(y,x)& D\Sb_N(x,y)\\
		\end{bmatrix}.\]
		As we have that
		\[\det\left(-\begin{bmatrix}0 & 1\\ 1 & 0 \end{bmatrix}\right)=1,\]
		operating on the Pfaffian kernel in this way doesn't change the correlation functions, and as applying this to (\ref{eqn:odd-temp-2}) yields (\ref{eqn:Pfaffian K-def}) with switched order, this completes the proof.
		
		\textbf{Case: $\tau=0$ and $N$ even}
		This essentially follows from the case of $\tau\neq 0$ by sending $\tau\to 0$ and using that $C_n(z)\to z^n$. Indeed, as noted in \cite{forrester-elliptic}, their formulas (particularly their (4.2)-(4.6)) which give $\bar{K}_N$ in terms of $\bar{S}_N$, hold regardless of $\tau$. What is left is to simplify their initial expression for $\bar{S}_N$ to ours, using the skew-orthogonal polynomials in their Corollary 2.  However, such a simplification for $\bar{S}_N$ is given in Theorem 8 of \cite{borodin-elliptic}. Indeed, it is clear that the first terms in both our and their definitions of $\bar{S}_N$ coincide. For second term, we need to show that 
		\[\phi_{N}(x):=\left(\frac{\sq{2\pi}N!}{2^{N/2}\Gamma(N/2+1)}\right)-2\int_x^\infty t^{N}e^{-\frac{t^2}{2}} dt=2\int_0^x t^Ne^{\frac{-t^2}{2}}dt.\label{eqn:ignore-2983094823-even}\]
		which follows from the Gaussian moment identity
		\[\int_{-\infty}^{\infty} e^{-\frac{t^2}{2}}|t|^{N} dt=\left(\frac{\sq{2\pi}N!}{2^{N/2}\Gamma(N/2+1)}\right).\label{eqn:ignore-gmi}\]
		
		\textbf{Case: $\tau=0$ and $N$ odd}
		In this case, the odd formula formula for $\bar{S}_N$ is also derived from the above skew-orthogonal polynomials in Proposition 2.2 of \cite{ginibre-odd}, using the methods of \cite{sinclair}, and a similar argument as in the $\tau\neq 0$ case establishes that $\bar{K}_N$ is of the desired form.
	\end{proof}
	
	We now proceed to the proof of Lemma \ref{lem:error estimate for Pfaffian kernel}, beginning with the case of $\tau \neq 0$. Our first step will be the following global bound on the Hermite polynomials.
	
	\begin{lem}
		\label{lem:technical:Hermite upper bound}
		There is $C>0$ such that for all $k\ge 1$ and $x\in \R$
		\[|H_k(\sq{2k}x)|\le
		C(k!)^{1/2}2^{k/2}e^{k s(x)},\]
		where here
		\[s(x):=x^2-I(|x|\ge 1)\left(|x|\sq{x^2-1}-\log(\sq{x^2-1}+|x|)\right)=\]
		\[\frac{1}{2}+\log(2)+\frac{2}{\pi}\int_{-1}^1\sq{1-y^2}\log(|x-y|)dy.\]
	\end{lem}
	\begin{proof}
		Cram\'{e}r's Inequality (see \cite{szego}) states that for any $x\in \R$ and $k\ge 0$
		\[|H_k(\sq{2k}x)|\le (k!)^{1/2}2^{k/2}e^{k x^2}.\]
		As $s(x)-x^2=0$ for $|x|<1$ and $s(x)-x^2=O(||x|-1|)$ elsewhere, this gives the desired claim for $|x|<1+k^{-1}$. For the remaining regions, we need to employ the Plancherel-Rotach asymptotics. In particular, we may employ the asymptotic expansions in Theorem 1 and Corollary 2 of \cite{diego-pr-asym} to find $C$ such that for $|x|>1$
		\[|H_k(\sq{2k}x)|\le Ck^{k/2}e^{-k/2}2^{k/2}e^{ks(x)}\sq{\left(1-|x|/\sq{x^2-1}\right)}\label{eqn:ignore-121}\]
		When $|x|\ge 1+k^{-1}$, the last term is of order $O(k^{1/4})$, so employing (\ref{eqn:ignore-121}) and Stirling's approximation completes the proof.
	\end{proof}
	
	From this we obtain the immediate corollary.
	
	\begin{corr}
		\label{corr:technical:orthogonal polynomial upper bound}
		There is $C>0$ such that for all $k\ge 1$ and $x\in \R$
		\[|e^{-\frac{kx^2}{2(1+\tau)}}C_k(\sq{k}x)|\le C (k!)^{1/2}e^{k h_\tau(x)}.\label{eqn:ignore-204}\]
		where here
		\[h_\tau(x):=-\frac{x^2}{2(1+\tau)}+\frac{1}{2}\log(\tau)+s\left(\frac{x}{2\tau}\right).\]
	\end{corr}
	
	It is easily verified that $h_\tau(x)$ is strictly maximized at the points $\{\pm(1+\tau)\}$, where $h_\tau(\pm(1+\tau))=0$. From this we see that for $\epsilon>0$ we may choose $c,C>0$ such that if both $|x-(1+\tau)|>\epsilon$ and $|x+(1+\tau)|>\epsilon$
	\[|e^{-\frac{kx^2}{2(1+\tau)}}C_k(\sq{k}x)|\le C (k!)^{1/2}e^{-ck}.\label{eqn:ignore-214}\]	
	
	\begin{proof}[Proof of Lemma \ref{lem:error estimate for Pfaffian kernel}]
		\textbf{Case: $\tau\neq 0$}
		
		Our first step is to employ Mehler's formula (see \cite{AS}) which shows that for $|\rho|<1$,
		\[\sum_{k=0}^{\infty}\frac{\rho^k}{2^k k!}H_{k}(x)H_k(y)=\frac{1}{\sq{1-\rho^2}}\exp \left(-\frac{\rho^2(x^2+y^2)-2\rho xy}{(1-\rho^2)}\right).\label{eqn:mehler}\]
		In particular, we observe that
		\[\frac{e^{-\frac{x^2+y^2}{2(1+\tau)}}}{\sq{2\pi}}\sum_{k=0}^{\infty}\frac{1}{k!}C_k(x)C_k(y)=\cal{S}_{\tau}(x,y).\]
		Thus, we may write
		\[S_N(x,y)-\cal{S}_{\tau}(\sq{N}x,\sq{N}y)=-\frac{e^{-\frac{N\left(x^2+y^2\right)}{2(1+\tau)}}}{\sq{2\pi}}\sum_{k=N-1}^{\infty}\frac{1}{k!}C_k(\sq{N}x)C_k(\sq{N}y)+\]\[\frac{e^{-N\frac{y^2}{2(1+\tau)}}}{2\sq{2\pi}(1+\tau)}\frac{C_{N-1}(\sq{N}y)\phi_{N-2}(\sq{N}x)}{(N-2)!}.\label{eqn:ignore-203}\]
		For the remainder of the proof we will fix a choice of $\epsilon>0$, and consider $x,y\in (-1-\tau+\epsilon,1+\tau-\epsilon)$ unless otherwise stated. We also allow the constants $C,c>0$ (which are dependent on $\epsilon$) to change from line to line. From (\ref{eqn:ignore-214}) we see that
		\[\l|e^{-\frac{N\left(x^2+y^2\right)}{2(1+\tau)}}\sum_{k=N-1}^{\infty}\frac{1}{k!}C_k(\sq{N}x)C_k(\sq{N}y)\r|\le C\sum_{k=N-1}^{\infty}e^{-2kc}\le Ce^{-Nc}.\]
		We now focus our attention on the second term in (\ref{eqn:ignore-203}). We note that employing (\ref{eqn:ignore-204}),
		\[\left|\int_x^\infty e^{-t^2/2(1+\tau)}C_N(t)dt\right|\le C(N!)^{1/2}.\]
		Further employing Stirling's formula to the constant term in (\ref{eqn:ignore-phi-def}) we see that 
		\[|\phi_N(x)|\le C(N!)^{1/2}. \label{eqn:ignore-215}\] 
		Combining these, we conclude that
		\[\left|e^{-N\frac{y^2}{2(1+\tau)}}\frac{C_{N-1}(\sq{N}y)\phi_{N-2}(\sq{N}x)}{(N-2)!}\right|\le Ce^{-Nc}.\label{eqn:ignore-129823748728}\]
		Altogether, this establishes that
		\[|S_N(x,y)-\cal{S}_N(\sqrt{N}x,\sqrt{N}y)|\le Ce^{-Nc}.\]
		To understand the derivatives, we may repeatedly employ the recurrence relation
		\[H'_k(x)=2xH_k(x)-H_{k+1}(x),\]
		to obtain from (\ref{eqn:ignore-214}) that there is $C>0$ such that for $x\in \R$
		\[\l|\frac{d^\ell}{dx^\ell}\left(e^{-\frac{k}{2(1+\tau)x^2}}C_k(\sq{k}x)\right)\r|\le C k^\ell (k!)^{1/2}e^{k h_\tau(x)}.\]
		Repeating the above analysis, we see that we have for any $n,m\ge 0$
		\[\l|\frac{d^n}{dx^n}\frac{d^m}{dy^m}\left(S_N(x,y)-\cal{S}(\sq{N}x,\sq{N}y)\right)\r|\le Ce^{-Nc}.\]
		Employing (\ref{eqn:Pfaffian I and D def}), (\ref{eqn:Pfaffian pure D}) and (\ref{eqn:Pfaffian pure I}) we see that same claim holds for $D_N(x,y)-\cal{D}(\sq{N}x,\sq{N}y)$ and $I_N(x,y)-\cal{I}(\sq{N}x,\sq{N}y)$, which suffices to demonstrate the desired claim. 
		
		\textbf{Case: $\tau= 0$} When $\tau=0$, instead of Mehler's formula, we are simply using the relation
		\[e^{-\frac{x^2+y^2}{2}}\sum_{k=0}^\infty\frac{x^ky^k}{k!}=\exp\left(-\frac{(x-y)^2}{2}\right).\]
		Moreover, the modification appropriate modification of the inequality in Corollary \ref{corr:technical:orthogonal polynomial upper bound} is the stronger equality statement: For all $k\ge 1$ and $x\in \R$ we have that
		\[e^{-kx^2/2}|\left(\sqrt{2k}x\right)^k|= k^{k/2}2^{k/2}e^{kh_0(x)},\]
		where here
		\[h_0(x)=-\frac{x^2}{2}+\log(|x|).\]

		As before, this $h_0$ is strictly maximized at the points $x=\pm 1$, where it takes the value zero. So for $\epsilon>0$, we conclude there are $C,c>0$ such that if $|x-1|>\epsilon$ and $|x+1|>\epsilon$ we have that
		\[|\left(e^{-kx^2/2}\sqrt{2k}x\right)^k|\le Ce^{-ck}.\]
		This is the modification of (\ref{eqn:ignore-214}). Now substituting these modifications in the proof of the $\tau\neq 0$ case provides a proof in this case, finishing the proof.
	\end{proof}
	
	\appendix
	
	\section{Conjecture \ref{conjecture: moment conjecture} and Fractional Moments \label{section:appendix:conjecture}}
	
	Here we will support Conjecture \ref{conjecture: moment conjecture}, as well as show that the formula for $C_{\tau}(\ell)$ given in Conjecture \ref{conjecture: moment conjecture} coincides with the expression given when $\ell\in \N$ in Theorem \ref{theorem: moments of real bulk}. Our proof is based on asymptotics for the Selberg integral, as used by Forrester and Frankel \cite{forrester-fischer-hartwig-asymptotics} to derive their many famed conjectures on the non-integral (and non-Hermitian) asymptotics for moments of the characteristic polynomials coming from a number of different families.
	
	The first key support is that relation
	\[C_{\tau}(\ell)=\left(\frac{(1+\tau)^{\ell/2}}{(1-\tau^2)^{\ell(\ell+1)/4}}\right)C_0(\ell).\]
	As this holds for all integral $\ell$, and its expression is quite simple, it is quite reasonable to conjecture that it will hold for all non-integral values of $\ell$. Assuming as well that there is a limiting constant which remains independent of the choice of point in the bulk, one would only need to show this in the case of $\tau=0$ and $x=0$. However, Conjecture \ref{conjecture: moment conjecture} is easily obtained in this case, as we now show.
	
	For this we note that the expression used to study the integral moments above (i.e. \ref{eqn:determinant of ginibre ratio}) is still valid for non-integral powers. In particular, as long as $\ell>-1$ and $N$ is even we have that
	\[\E|\det(A_{N,0})|^\ell=N^{-N\ell/2}2^{N\ell/2}\prod_{i=1}^{N}\frac{\Gamma((\ell+i)/2)}{\Gamma(i/2)}=\]
	\[\left(\frac{N}{2}\right)^{-N\ell/2}\left(\prod_{i=1}^{N/2}\frac{\Gamma(\ell/2-1/2+i)}{\Gamma(-1/2+i)}\right)\left(\prod_{i=1}^{N/2}\frac{\Gamma(\ell/2+i)}{\Gamma(i)}\right)\label{eqn:ignore-2347889234}\]
	Product of this type can be understood using the recurrence formula for the $G$-function (see (24) of \cite{forrester-constant}) so that we may rewrite this product as
	\[\left(\frac{N}{2}\right)^{-N\ell/2}\left(\frac{G(N/2+\ell/2+1/2)G(1/2)}{G(\ell/2+1/2)G(N/2+1/2)}\right)\left(\frac{G(N/2+\ell/2+1)}{G(\ell/2+1)G(N/2+1)}\right).\]
	Employing the asymptotic formula
	\[\frac{G(N+a+1)}{G(N+b+1)}=\left(\frac{e^{N}}{N(2\pi)^{1/2}}\right)^{b-a}N^{\left(a^2-b^2\right)/2}(1+o(1)).\]
	Applying this repeatedly we obtain that (\ref{eqn:ignore-2347889234}) is
	\[e^{-\frac{N\ell}{2}}(2\pi)^{\ell/2}\left(\frac{N}{2}\right)^{\ell(\ell-1)/4}\left(\frac{G(1/2)}{G(\ell/2+1/2)G(\ell/2+1)}\right)(1+o(1)),\]
	which gives Conjecture \ref{conjecture: moment conjecture} in the case of $x=0$, $\tau=0$ and $N$ even.  A similar proof establishes the case of $N$ odd. In particular, this confirms that for integral $\ell$ we have that
	\[C_0(\ell)=\frac{(2\pi)^{\ell/2}G(1/2)}{2^{\ell(\ell-1)/4}G(\ell/2+1/2)G(\ell/2+1)},\]
	as claimed.
	\bibliographystyle{abbrv}
	\bibliography{main}
\end{document}